\documentclass{article}
\usepackage{amssymb,amsmath,amsthm,verbatim}
\usepackage{enumerate}
\usepackage{tikz}

\newcommand{\GA}{{\rm GA}}
\newcommand{\MA}{{\rm MA}}

\newcommand{\BA}{{\rm BA}}
\newcommand{\TA}{{\rm TA}}

\newcommand{\Af}{{\rm Aff}}

\DeclareMathOperator{\ldeg}{ldeg}

\DeclareMathOperator{\Spec}{Spec}
\DeclareMathOperator{\id}{id}

\newcommand{\A}{\mathbb{A}}

\newcommand{\IN}{\mathbb{N}}

\newcommand{\IR}{\mathbb{R}}

\newcommand{\Ik}{\mathbb{K}}

\newcommand{\s}{\sigma}

\newcommand{\al}{\alpha}
\newcommand{\bx}{{\bf x}}

\newtheorem{theorem}{Theorem}
\newtheorem{proposition}[theorem]{Proposition}
\newtheorem{lemma}[theorem]{Lemma}

\newtheorem{corollary}[theorem]{Corollary}
\newtheorem{conjecture}{Conjecture}
\newtheorem{question}[conjecture]{Question}

\theoremstyle{definition}
\newtheorem{definition}{Definition}

\theoremstyle{remark}
\newtheorem{remark}{Remark}
\newtheorem{example}{Example}

\newcommand{\New}{\ensuremath{{\rm New}}}

\newcommand{\supp}{{\ensuremath{\rm supp\ }} }

\newcommand{\conv}{\ensuremath{{\rm conv}}}

\newcommand{\cotame}{co-tame}

\newcommand{\pri}{\ensuremath{\smallsetminus}}

\title{The affine automorphism group of ${\mathbb A}^3$ is not a maximal subgroup of the tame automorphism group}

\author{%
Eric Edo\thanks{ERIM, University of New Caledonia. Email address: \texttt{eric.edo@univ-nc.nc}} and %
Drew Lewis\thanks{Department of Mathematics,  University of Alabama.  Email address: \texttt{dlewis@ua.edu}}%
}

\begin{document}
\maketitle

\begin{abstract}
We construct explicitly a family of proper subgroups of the tame automorphism group of affine three-space (in any characteristic)
which are generated by the affine subgroup and a non-affine tame automorphism.  One important corollary is the titular result that settles negatively the open question (in characteristic zero) of whether  the affine subgroup is a maximal subgroup of the tame automorphism group.  We also prove that all groups of this family have the structure of an amalgamated free product of the affine group and a finite group over their intersection.
\end{abstract}

\section{Introduction}

Throughout, $\Ik$ denotes a field of any characteristic.
We denote by $\GA_n(\Ik)$ the group of polynomial automorphisms of $\A^n_{\Ik}$.
We consider ${\rm Aff}_n(\Ik)$ (resp. ${\rm BA}_n(\Ik)$, resp. ${\rm TA}_n(\Ik)$), the subgroup of
$\GA_n(\Ik)$ of affine (resp. triangular, resp. tame) automorphisms (see Section~2 or \cite{vdE} for precise definitions).
In this paper we are interested with the question of finding proper intermediate subgroups between
${\rm Aff}_n(\Ik)$ and ${\rm TA}_n(\Ik)$.

If $n=2$, it is well known that such intermediate subgroups exist.  The classical Jung-van der Kulk theorem (\cite{Jung,Kulk}) states that
${\rm GA}_2(\Ik)={\rm TA}_2(\Ik)$ and, moreover, $\GA_2(\Ik)$ is the amalgamated free product of ${\rm Aff}_2(\Ik)$ and ${\rm BA}_2(\Ik)$ along their intersection.
Using this structure theorem, we can uniquely define the height of any automorphism $\phi \in \GA_2(\Ik)$ as the maximum of the degrees of the triangular automorphisms in any reduced decomposition of  $\phi$, and let $H_d$ denote the set of all automorphisms of height at most $d$.  Then we have ${\rm Aff_2}(\Ik)=H_1 \subset H_2 \subset H_3 \subset\cdots \subset \TA_2(\Ik)$ is an ascending sequence of (proper) subgroups of ${\rm TA}_2(\Ik)$.
In particular, for all $\beta\in{\rm BA}_2(\Ik)\pri{\rm Aff}_2(\Ik)$ then $\langle{\rm Aff}_2(\Ik),\beta\rangle$
is a proper subgroup of ${\rm TA}_2(\Ik)$.

In the case that $n > 2$ and $\Ik$ has positive characteristic, then it is also known that there are many intermediate subgroups between ${\rm Aff}_n(\Ik)$
and ${\rm TA}_n(\Ik)$ (see, for example, \cite{EK}).  However, in characteristic zero, the question is much more nuanced\footnotemark.  The first partial results in this direction concern subgroups of the form $\langle \Af_n(\Ik), \beta \rangle$ for a single automorphism $\beta \in \GA_n(\Ik) \setminus {\rm Aff _n}(\Ik)$.  In 1997, Derksen gave an elementary proof (unpublished, but see \cite{vdE} Theorem 5.2.1 for a proof) that the triangular automorphism $\s:=(x_1+x_2^2,x_2,\ldots,x_n)\in{\rm BA}_n(\Ik)$, along with the affine subgroup, generates the entire tame group (when ${\rm char}(\Ik)=0$); that is,  $\langle{\rm Aff}_n(\Ik),\s\rangle={\rm TA}_n(\Ik)$.  This motivated the definition of {\em \cotame} automorphisms as follows:
\begin{definition}
An automorphism $\phi\in\GA_n(\Ik)$ is called {\em \cotame} if $\langle{\rm Aff}_n(\Ik),\phi\rangle\supset{\rm TA}_n(\Ik)$.
\end{definition}
\footnotetext{Recently, Wright \cite{Wright} showed that in characteristic zero, $\TA_3 (\Ik)$ is an amalgamated free product of three subgroups along their pairwise intersection, which implies a much weaker structure on $\TA_3(\Ik)$.  Unlike in dimension two, we no longer have a reasonably unique representation of every tame automorphism. }

One can naturally ask:

\begin{question}\label{Q:cotame}Let $n \geq 3$, and let $\Ik$ be a field of characteristic zero. Is every automorphism in $\GA_n(\Ik)$ \cotame?
\end{question}

Note that this is intimately related to the question of finding intermediate subgroups, as an example of an automorphism $\beta$ which is tame but not \cotame\ would provide an intermediate subgroup ${\rm Aff}_n(\Ik) \subset \langle {\rm Aff }_n (\Ik), \beta \rangle \subset \TA_n(\Ik)$.

In 2004, Bodnarchuk \cite{Bodnarchuk} generalized Derksen's result in the following way: if $\Ik$ has characteristic zero,
then all non-affine triangular and bitriangular automorphisms (i.e.,  elements of the form $\beta _1 \alpha \beta _2$ for some $\beta _1, \beta _2 \in \BA_n(\Ik)$ and $\alpha \in \Af _n(\Ik)$) are \cotame.  Interestingly, the first author \cite{E} recently showed that certain wild (i.e., not tame) automorphisms, including the famous Nagata automorphism, are \cotame.

In this paper, we provide a negative answer to Question \ref{Q:cotame} when $n=3$ by constructing an automorphism which is tame but not \cotame.  More precisely, fix an integer $N\ge 1$ and
consider the automorphisms $\beta=(x+y^2(y+z^2)^2,y+z^2,z)\in \BA_3(\Ik)$, $\pi=(y,x,z)\in \Af_3(\Ik)$ and
$\theta_N=(\pi\beta)^N\pi(\pi\beta)^{-N} \in \TA_3(\Ik)$. We prove the following result (without any assumption about the characteristic of $\Ik$):\\

\noindent {\bf Main Theorem.} \textit{For all integers $N\ge 3$, the automorphism $\theta_N$ is not \cotame.
In other words, $\langle{\rm Aff}_3(\Ik),\theta_N\rangle$ is a proper subgroup of ${\rm TA}_3(\Ik)$.
Moreover, this group is the amalgamated free product of ${\rm Aff}_3(\Ik)$ and
$\langle{\cal C},\theta_N\rangle$ along their intersection ${\cal C}$ where ${\cal C}=\{\al\in{\rm Aff}_3(\Ik)\ |\ \al\theta_N=\theta_N\al\}$ is a finite cyclic group.\\
}

\begin{remark} $\theta _1$ is \cotame\ by the aforementioned result of Bodnarchuk.
\end{remark}

The main theorem immediately implies the result in the title of this paper:
\begin{corollary}
For any field $\Ik$, ${\rm Aff} _3 (\Ik)$ is not a maximal subgroup of $\TA _3 (\Ik)$.
\end{corollary}



In section \ref{s:general}, we describe some general, commonly used definitions.  We make some specific notations and definitions in section \ref{s:notations} which are necessary to state our key technical result (Theorem \ref{thm:Pstable}).  This statement of Theorem \ref{thm:Pstable} and the proofs of its consequences (including the main theorem) comprise  section \ref{s:main}.  The proof of Theorem \ref{thm:Pstable} is quite technical, and is deferred to the ultimate section.

\section{General definitions}\label{s:general}
\subsection{Degrees}
Let $n\ge 1$ be an integer.
We denote by $\Ik[\bx]=\Ik[x_1,\ldots,x_n]$ the polynomial algebra in $n$ commutative variables $\bx=\{x_1,\ldots,x_n\}$.
We write $\bx^v=x_1^{i_1}\cdots x_n^{i_n}$ for any $v=(i_1,\ldots,i_n)\in\IN^n$.
For a given $P\in \Ik[\bx]$ we denote by $\supp(P)\subset\IN^n$ the \textit{support} of $P$; that is, the
set of $n$-tuples $v\in\IN^n$ such that the coefficient of $\bx^v$ in $P$ is nonzero.

The main technical tool in this paper is the use of various degree functions.  Here, we mean ``degree function'' in a little more generality than most authors, so we give a precise definition.  Typically, the co-domain of a degree function is the natural numbers or the integers; we instead consider any totally ordered commutative monoid $M$, and set $\overline{M}=M \cup {-\infty}$ with the convention that $-\infty+n=-\infty$ and $-\infty < n$ for all $n \in M$.
\begin{definition}Let $A$ be a $\Ik$-domain and $M$ a totally ordered commutative monoid.
A map $\deg : A \rightarrow \overline{M}$ is called a {\em degree function} provided that
\begin{enumerate}
\item $\deg(f)=-\infty$ if and only if $f=0$,
\item $\deg(fg)=\deg(f)+\deg(g)$  for all $f,g \in A$,
\item $\deg(f+g) \leq \max \{ \deg(f), \deg(g)\}$ for all $f,g \in A$.
\end{enumerate}
\end{definition}

Two easy consequences of the definition are that $\deg(c)=0$ for any $c \in \Ik^*$, and that if $\deg(f) \neq \deg(g)$, then equality holds in property 3.

The two families of degree functions that we will use are {\em weighted degree} and {\em lexicographic degree}, the latter of which takes values in $\IN^n$.
\begin{itemize}
\item
For any $w\in\IN^n\pri\{\bf 0\}$, we denote the $w$\textit{-weighted degree} of $P$ by $$ \deg_w(P)=\max _{v \in \supp(P)} \{v \cdot w\},$$ where ($\cdot$) denotes the scalar product
in $\IR^n$.  The case $w=(1,\ldots,1)$ corresponds to the usual notion of the total degree of a polynomial.

\item
For an integer $1\le i\le n$, let ${\geq} _i$ denote the $i$-th cyclic lexicographic ordering of $\IN^n$; that is, the standard basis vectors are ordered by $$e_i >_i e_{i+1} >_i \cdots >_i e_n >_i e_1 >_i \cdots >_i e_{i-1}.$$  
Letting ${\rm max} _i$ denote the maximum with respect to this ordering, we define the {\em $i$-th lexicographic degree} of $P$ to be
 $${\rm ldeg}_i(P) = {\rm max} _i (\supp(P))$$

\end{itemize}

\begin{example} Let $P=x+y^2(y+z^2)^2$.  Then we have 
\begin{align*}
 \deg _{(4,1,0)} (P) &= 4  & \ldeg _1 (P) &= (1,0,0) \\
\deg _{(4,0,1))} (P) &= 4 & \ldeg _2 (P) &= (0,4,0) \\
\deg _{(8,2,1)} P &=8 & \ldeg _3 (P) &= (0,2,4).
\end{align*}
\end{example}


\subsection{Polynomial Automorphisms}
We adopt the following standard notations of polynomial automorphism groups:
\begin{itemize}
\item $\MA_n(\Ik)$ denotes the monoid of polynomial endomorphisms; that is, the set $\Ik[\bx]^n$ with the composition $$(\phi_1,\ldots,\phi_n)(\psi_1,\ldots,\psi_n)=(\phi_1(\psi_1,\ldots,\psi_n),\ldots,\phi_n(\psi_1,\ldots,\psi_n)).$$
\item $\GA_n(\Ik)$ is the group of polynomial automorphisms (or the {\em general automorphism group}), defined to be the group of invertible elements of $\MA_n(\Ik)$.
\item The \textit{affine subgroup} of $\GA_n(\Ik)$ is
$$\Af_n(\Ik)=\{(\phi_1,\ldots,\phi_n)\in\GA_n(\Ik)\,|\,\deg_{(1,\ldots,1)}(\phi_i)=1\ \text{for each}\ 1\le i\le n\}.$$
\item The \textit{triangular subgroup} of $\GA_n(\Ik)$ is
$$\BA_n(\Ik)=\{(\phi_1,\ldots,\phi_n)\in\GA_n(\Ik)\,|\,\phi_i\in\Ik^*x_i+\Ik[x_{i+1},\ldots, x_n]\ \text{for each}\ 1\le i\le n\}$$
\item The \textit{tame subgroup} is $\TA_n(\Ik) = \langle \Af_n(\Ik), \BA_n(\Ik) \rangle$.  It is well known to be the entire group $\GA_n(\Ik)$ for $n=1,2$, while Shestakov and Umirbaev \cite{SU} famously showed that it is a proper subgroup when $n=3$ and ${\rm char} (\Ik)=0$.  Whether it is a proper subgroup or not is a well known, quite difficult open question in higher dimensions and/or positive characteristic.
\end{itemize}

The group $\GA_n(\Ik)$ is isomorphic to the group of automorphisms of $\Spec \Ik[\bx]$ over $\Spec \Ik$, and is anti-isomorphic to the group of $\Ik$-automorphisms of $\Ik[\bx]$.
We freely abuse this correspondence and, given $\phi\in \GA_n(\Ik)$ and $P \in\Ik[\bx]$, we denote by $(P)\phi \in \Ik[\bx]$ the image of $P$ by the $\Ik$-automorphism of $\Ik[\bx]$ corresponding to $\phi$.
By writing the automorphism on the right, the expected composition holds,
namely $(P)\phi \psi = ((P)\phi)\psi$ for $P\in\Ik[\bx]$ and $\phi, \psi \in \GA_n(\Ik)$.  We refer the reader to \cite{vdE} for a comprehensive reference on polynomial automorphisms.

We make one elementary observation on how one can compute the degree of the image of a polynomial under an automorphism.  We will use this frequently and without further mention.
\begin{lemma}
Let $\gamma\in\GA_n(\Ik)$, and let $P\in\Ik[\bx]$.
Let $\deg: \Ik[\bx] \rightarrow \overline{M}$ denote a degree function (for some totally ordered commutative monoid $M$), and let $m  \in M$.
If $\deg((\bx^v)\gamma)\le m$ for all $v\in\supp(P)$, then $\deg((P)\gamma)\le m$.
Moreover, if $\deg((\bx^v)\gamma)=m$ for a unique $v\in\supp(P)$ then $\deg((P)\gamma)= m$.
\end{lemma}

\section{Notations}\label{s:notations}

For the remainder of this paper, we restrict our attention to dimension 3.
For convenience, we set $\bx=\{x,y,z\}$ instead of $\{x_1,x_2,x_3\}$. We denote by ${\cal A}={\rm Aff}_3(\Ik)$
(resp. ${\cal B}=\BA_3(\Ik)$) the subgroup of affine (resp. triangular) automorphisms. 

We fix also an integer $N\ge 3$ and
we consider the following automorphisms:
\begin{align*}
\beta & =(x+y^2(y+z^2)^2,y+z^2,z)\in {\cal B} \\
\pi & =(y,x,z)\in {\cal A} \\
\theta & =\theta_N= (\pi\beta)^N\pi(\pi\beta)^{-N}\in \TA_3(\Ik).
\end{align*}

\begin{remark} We will repeatedly make use of the fact that $\pi$ and $\theta$ are involutions, i.e. $\beta^2=\pi^2=\id$.
\end{remark}

\subsection{Some sets of polynomials}

Let $m\ge 1$ and $n\ge 0$ be integers.  The following technical definitions will play a crucial role in our methods (see figures \ref{figp} and \ref{figq}):

\begin{align*}
P_{m,n}              &= \{(i,j,k) \in \IN^3\ |\ 4i+j \leq 4m,\ 4i+k \leq 4m+n,\ 8i+2j+k \leq 8m+n\} \\
\mathcal{P}_{m,n}    &= \{P \in\Ik[\bx]\ |\ \supp(P)\in P_{m,n}\} \\
                    &=\{P \in\Ik[\bx]\,|\,\deg_{(4,1,0)}(P)\leq 4m,\ \deg_{(4,0,1)}(P) \leq 4m+n,\ \deg_{(8,2,1)}(P)\leq 8m+n\} \\
\mathcal{P}_{m,n} ^* &= \{P \in \mathcal{P}_{m,n}\ |\ {\rm ldeg}_2(P)=(0,4m,n),\ {\rm ldeg}_3(P)=(0,2m,4m+n)\} \\
Q_{m,n}              &= \{(i,j,k) \in \IN^3\ |\ i+j \leq m,\ 3i+3j+k \leq 3m+n\} \\
\mathcal{Q}_{m,n}    &= \{P \in\Ik[\bx]\ |\ \supp(P)\in Q_{m,n}\} \\
                     &= \{P \in\Ik[\bx]\ |\ \deg_{(1,1,0)} \leq m,\ \deg_{(3,3,1)} \leq 3m+n \} \\
\mathcal{Q}_{m,n} ^* &= \{P \in \mathcal{Q}_{m,n}\ |\ {\rm ldeg}_2(P)=(0,m,n)\} 
\end{align*}

Moreover, we consider
$$\mathcal{P}^*=\bigcup_{m\ge 1,n\ge 0}\mathcal{P}_{m,n} ^*\hspace{.3cm}{\rm and}\hspace{.3cm}
\mathcal{Q}^*=\bigcup_{m\ge 1,n\ge 0}\mathcal{Q}_{m,n} ^*$$
For all integers $m\ge 1$ and $n\ge 0$, one can easily check that $\mathcal{P}_{m,n}^*\subset \mathcal{Q}_{4m,n}^*$, and thus $\mathcal{P}^*\subset \mathcal{Q}^*$.

\begin{example} $(x^mz^n)\beta \in \mathcal{P}_{m,n}^*$ for all $m \geq 1$, $n \geq 0$.
\end{example}

\begin{figure}[h]
\begin{center}

\begin{tikzpicture}[line join = round, line cap = round,scale=0.75, every node/.style={transform shape}]
\pgfmathsetmacro{\m}{1.4};
\pgfmathsetmacro{\n}{1.6};
\coordinate [label=right:{$(0,4m,n)$}] (A) at (4*\m,\n,0);
\coordinate  [label=below :{$(0,4m,0)$}] (B) at (4*\m,0,0);
\coordinate  [label=left:{$(m,0,n)$}](C) at (0,\n,\m);
\coordinate  [label=left:{$(m,0,0)$}] (D) at (0,0,\m);
\coordinate [label=above right:{$(0,2m,4m+n)$}](E) at (2*\m, 4*\m+\n,0);
\coordinate  [label=left: {$(0,0,4m+n)$}] (F) at (0,4*\m+\n,0);

\draw[-] (A) -- (B) -- (D) -- (C) --cycle;
\draw[label=left: {test}] (A) -- (C) -- (E) -- cycle;
\draw[-] (C) -- (E) -- (F) --cycle;

\node[draw=none] at (2*\m,\m+\n,0) {  $8i+2j+k = 8m+n$};
\draw[-latex] (\m,4.5*\m+\n,0) to[out=225,in=90,looseness=1] (0.5*\m, 4*\m+0.5*\n, 0) ;
\node[draw=none] at (1.5*\m, 4.6*\m+\n,0) {\small $4i+k=4m$};

\draw[-latex] (3*\m, -1/2*\n,0) to [out=135,in=270,looseness=1] (2*\m,\n /2,0);
\node[draw=none] at (3.5*\m,-1.3*\n /2,0) {$4i+j = 4m$};

\draw[->] (0,0) -- (5*\m,0,0) node[right] {$j$};
\draw[->] (0,0) -- (0,5*\m+\n,0) node[above] {$k$};
\draw[->] (0,0) -- (0,0,2*\m) node[below left] {$i$};

\end{tikzpicture}
\caption{ $P_{m,n}$\label{figp} }
\end{center}
\end{figure}

\begin{figure}
\begin{center}
\begin{tikzpicture}[line join = round, line cap = round, scale=0.75, every node/.style={transform shape}]
\pgfmathsetmacro{\m}{2};
\pgfmathsetmacro{\n}{1.5};
\coordinate [label=right:{$(0,m,n)$}] (A) at (\m,\n,0);
\coordinate  (B) at (\m,0,0);
\coordinate  [label=left:{$(m,0,n)$}](C) at (0,\n,\m);
\coordinate  (D) at (0,0,\m);
\coordinate  [label=above right:{$(0,0,3m+n)$}] (F) at (0,3*\m+\n,0);

\draw[-] (A) -- (B) -- (D) -- (C) --cycle;
\draw[-] (A) -- (C) -- (F) -- cycle;


\draw[-latex] (\m,\m+\n,0) to[out=180,in=90,looseness=1] (0.5*\m, \m+0.5*\n, 0) ;
\node[draw=none] at (1.9 *\m,\m+\n,0) { $3i+3j+k=3m+n$};

\draw[-latex] (\m, 0, 0.75*\m) to [out=180,in=270,looseness=1] (0.3*\m, 1/4*\n,0);
\node[draw=none] at (1.4*\m,0,0.75*\m) { $i+j=m$};

\draw[->] (0,0) -- (1.5*\m,0,0) node[right] {$j$};
\draw[->] (0,0) -- (0,3.5*\m+\n,0) node[above] {$k$};
\draw[->] (0,0) -- (0,0,1.5*\m) node[below left] {$i$};

\end{tikzpicture}
\caption{ $Q_{m,n}$\label{figq} }
\end{center}
\end{figure}

\begin{remark}These definitions are variations of a standard tool for studying polynomials, namely the Newton polytope.  The Newton polytope of a polynomial $P$ is defined as $\New (P) = \conv \left( \supp(P) \cup \{\bf 0\} \right)$ (here, $\conv$ denotes the convex hull in $\IR^3$).  For any $(a,b,c) \in \IN^3$, let ${\rm cub}(a,b,c)$ denote the rectangular cuboid $${\rm cub}(a,b,c)=\{(i,j,k) \in \IN^3\ |\ 0\leq i\leq a,\ 0\leq j \leq b,\ 0 \leq k \leq c\}.$$
For a set $S \subset \IN^3$, we set $${\rm cub}(S)=\bigcup _{(a,b,c) \in S} {\rm cub}(a,b,c).$$  Then we can interpret $P_{m,n}$ as $$P_{m,n}={\rm cub}\left( \supp \left( (x+y^2(y+z^2)^2)^mz^n \right) \right).$$
In particular, $\New \left(( (x+y^2(y+z^2)^2)^mz^n \right) \cap \IN^3 \subset P_{m,n}$.
\end{remark}

\subsection{Some subgroups of the affine group}

We consider the following nested sequence of subgroups of the affine group:
\begin{align*}
{\cal A}_0 &={\cal A} = {\rm Aff} _3 (\Ik),\\
{\cal A}_1 &={\cal A} \cap {\cal B} = {\rm Aff} _3 (\Ik) \cap  {\BA_3}(\Ik)  ,\\
{\cal A}_2 &=\{(u^8x+by+cz+d,u^2y,uz)\ |\ u\in\Ik^*,\,b,c,d\in\Ik\},\\
{\cal A}_3 &=\{(u^8x+cz+d,u^2y,uz)\ |\ u\in\Ik^*,\,c,d\in\Ik\},\\
{\cal A}_4 &=\{(u^2x,u^2y,uz)\ |\ u\in\Ik^*,u^6=1\}.
\end{align*}
If we set ${\cal C}=\{\al\in{\cal A}\ |\ \al\theta=\theta\al\}$, we have ${\cal A}_4\subset{\cal C}$
since for every element $\al\in{\cal A}_4$ it's easy to check that $\al\pi=\pi\al$ and $\al\beta=\beta\al$.
The opposite inclusion ${\cal C}\subset{\cal A}_4$ is a consequence of our main result.

\section{Main results}\label{s:main}

Using results from section~\ref{s:5props}, we prove our main technical result.
\begin{theorem}\label{thm:Pstable}
The set ${\cal P}^*$ is stable under the action of the automorphisms $\pi\beta$, $\pi\beta^{-1}$
and $(\pi\beta^{-1})^3\al\pi(\pi\beta)^3$ for any $\al\in{\cal A}\pri{\cal A}_4$.
\end{theorem}
\begin{proof}
If $\gamma\in\{\beta,\beta^{-1}\}$ then $({\cal P}^*)\pi\gamma\subset{\cal P}^*$ by Proposition~\ref{prop:PtoQ}.  To show that $(\pi\beta^{-1})^3\al\pi(\pi\beta)^3$ preserves ${\mathcal P}^*$, we separately consider the four cases $\al\in{\cal A}_{i-1}\pri{\cal A}_i$ ($i\in\{1,2,3,4\}$).  In particular, we have:

\begin{enumerate}[1)]
\item If $\alpha\in{\cal A}_0\pri{\cal A}_1$ then $(\mathcal{P}^*) \alpha \beta \subset \mathcal{Q}^*$ (Proposition~\ref{prop:A0-A1}).
\item If $\al\in{\cal A}_1\pri{\cal A}_2$ then $({\cal P}^*)\pi\beta^{-1}\al\beta\subset{\cal Q}^*$ (Proposition~\ref{prop:A1-A2}).
\item If $\al\in{\cal A}_2\pri{\cal A}_3$ then $({\cal P}^*)\pi\beta^{-1}\al\beta\pi\beta\subset{\cal Q}^*$ (Proposition~\ref{prop:A2-A3}).
\item If $\al\in{\cal A}_3\pri{\cal A}_4$ then $({\cal P}^*)(\pi\beta^{-1})^2\al\beta\pi\beta\subset{\cal Q}^*$ (Proposition~\ref{prop:A3-A4}).
\end{enumerate}
But by Proposition \ref{prop:PtoQ}, $({\mathcal Q}^*) \pi \beta \subset {\mathcal P}^*$, so we can simply apply $\pi \beta$ once more to end up in ${\mathcal P}^*$.  Thus $({\mathcal P}^*)(\pi \beta ^{-1} )^3 \alpha \pi (\pi \beta)^3 \subset {\mathcal P}^*$.
\end{proof}

\begin{remark} If $\alpha \in {\cal A}_4$, then $(\pi\beta^{-1})^3\al\pi(\pi\beta)^3 = \pi \alpha$, which does not preserve $\mathcal{P}^*$ in general.  However, since $\mathcal{A}_4 \subset \mathcal{C}$, we will take advantage of the commutativity of these elements to push them out of the way.
\end{remark}

\begin{corollary}\label{cor:Pstable1}
Let $r\ge 1$ be an integer. Let $\al_0,\ldots,\al_{r}\in{\cal A}$, and set $\phi=\al_0\theta\al_1\cdots\theta\al_{r}$.
If $\al_1,\ldots,\al_{r-1}\in{\cal A}\pri {\cal A}_4$, then there exist $\al,\al'\in {\cal A}$ such that
$(y)\al\phi\al'\in{\cal P}^*$.
\end{corollary}
\begin{proof}
We set $\theta^\prime=\theta\pi = (\pi\beta)^N (\pi \beta ^{-1})^N$;
$\alpha = \alpha _0 ^{-1}$; $\alpha ^\prime = \alpha _{r}^{-1}\pi$;
and $\alpha _i ^\prime = \pi \alpha _i  \pi $ for $1 \leq i \leq r-1$.  Note that since $\mathcal{A}_4$ is fixed under conjugation by $\pi$, $\alpha _1 ^\prime, \ldots, \alpha _{r-1} ^\prime \notin {\mathcal A}_4$.  Then we have
\begin{align*}
\alpha \phi \alpha ^\prime &=  \theta^\prime  \alpha _1 ^\prime \pi \theta^\prime \alpha _2^\prime \pi \cdots \theta^\prime \alpha _{r-1}^\prime \pi \theta^\prime \\
&=(\pi \beta) (\pi \beta) ^{N-1} \left( \prod _{i=1} ^{r-1} (\pi \beta ^{-1}) ^N \alpha _i ^\prime \pi (\pi \beta) ^N\right) (\pi \beta ^{-1}) ^N
\end{align*}
Since $(y)\pi \beta \in {\cal P}_{1,0}^{*}\subset{\cal P}^{*} $, we deduce $(y)\al\phi\al'\in{\cal P}^*$ by Theorem \ref{thm:Pstable}.
\end{proof}

\begin{corollary}\label{cor:Pstable2} Let $\phi\in\langle{\cal A},\theta\rangle\pri{\cal A}$.  Then there exist $\al,\al'\in {\cal A}$ such that
$(y)\al\phi\al'\in{\cal P}^*$.
\end{corollary}
\begin{proof}
Since $\theta^{-1}=\theta$ and $\phi\in\langle{\cal A},\theta\rangle$, we can write
$\phi = \alpha _0 \theta \alpha _1 \theta \cdots \alpha _r \theta \alpha _{r}$ for some $\alpha _0,\ldots,\alpha _{r} \in {\cal A}$.
If some $\alpha _i \in {\mathcal C}$ ($1 \leq i \leq r-1$), then we can shorten our sequence.  Thus, since $\phi \notin {\mathcal A}$, we may assume $\al_1,\ldots,\al_{r-1}\in{\cal A}\pri {\cal C}\subset {\cal A}\pri {\cal A}_4$ and $r\ge 1$, and hence the result follows from Corollary \ref{cor:Pstable1}
\end{proof}

\begin{corollary}\label{cor:com}
${\cal C}={\cal A}_4$.  In particular, ${\mathcal C}$ is a finite cyclic group of order $1,2,3$ or $6$.
\end{corollary}
\begin{proof}
We noted above that ${\cal A}_4\subset{\cal C}$, so we are left to prove the opposite containment.  Suppose for contradiction that there exists $\rho\in{\cal C}\pri{\cal A}_4$.
By Corollary~\ref{cor:Pstable1} (with $r=2$, $\al_0=\al_2={\rm id}$ and $\al_1=\rho$),
there exist $\al,\al'\in {\cal A}$ such that $(y)\al\theta\rho\theta\al'\in{\cal P}^*$.
Since $\rho$ commutes with the involution $\theta$, we have $\al\theta\rho\theta\al'=\al\rho\al'\in{\cal A}$.
This is a contradiction, as for any $\gamma \in {\mathcal A}$, $\ldeg _2 ( (y)\gamma) \leq_2 (0,1,0) <_2 (0,4m,n)$ for any $m \geq 1$, $n \geq 0$ and thus $(y)\gamma \notin {\mathcal P}^*$.
\end{proof}

\begin{corollary}
Let $\phi\in \TA_3(\Ik)$ be a tame automorphism with $\deg_{(1,1,1)}((f)\phi)\le 5$ for all $f \in \Ik[\bx]$ with $\deg _{(1,1,1)} (f) = 1$.  If $\phi \notin \mathcal{A}$, then $\phi\not\in\langle{\cal A},\theta\rangle$.  In particular, $ \langle \mathcal{A}, \theta \rangle$ is a proper subgroup of $\mathcal{T}$.
\end{corollary}

\begin{proof}
Suppose for contradiction that $\phi\in\langle{\cal A},\theta \rangle$. Applying Corollary~\ref{cor:Pstable2} to $\phi$, there exist $\al,\al'\in {\cal A}$ such that $(y)\al\phi\al'\in{\cal P}^*$. Since $\deg _{(1,1,1)} (y) \alpha = 1$, by assumption $\deg _{(1,1,1)} ( (y)\alpha \phi) \leq 5$, and so $\deg _{(1,1,1)} ( (y)\alpha \phi \alpha ^\prime ) \leq 5$ (since $\alpha ^\prime$ is affine).
But if $P=(y)\alpha \phi \alpha ^\prime \in{\cal P}^*$, there must exist integers $m\ge 1$ and $n\ge 0$ such that $P\in{\cal P}_{m,n}^*$,
in which case $(0,2m,4m+n)\in\supp(P)$.  Thus $5 \geq \deg_{(1,1,1)}(P)\ge 6m+n\ge 6$, a contradiction.  So we must have $\phi \not \in\langle{\cal A},\theta \rangle$.

\end{proof}

\begin{example} $(x+y^2,y,z) \in \TA_3(\Ik) \setminus \langle \mathcal{A}, \theta\rangle$.
\end{example}

\begin{corollary}
The group $\langle{\cal A},\theta\rangle$ is the amalgamated free product of ${\cal A}$ and
$\langle{\cal C},\theta\rangle$ along their intersection ${\cal C}$.
\end{corollary}

\begin{proof}
Let $r\ge 1$ be an integer.  Let $\al_1,\ldots,\al_{r-1}\in{\cal A}\pri\langle{\cal C},\theta\rangle$
and $\rho_1,\ldots,\rho_{r}\in\langle{\cal C},\theta\rangle\pri{\cal A}$.
Set $\phi:=\rho_1\al_1\ldots\al_{r-1}\rho_{r}$.  Since it is clear that $\mathcal{A}$ and $\langle \mathcal{C}, \theta \rangle$ generate $\langle \mathcal{A}, \theta \rangle$, it suffices to check that $\phi \notin \mathcal{C}$.
Using the fact that $\theta$ is an involution, we can write $\rho_i=\theta c_i$ for some $c_i\in{\cal C}={\mathcal A}_4$  for each $1\le i\le r$.
By  Corollary~\ref{cor:Pstable1}
(with $\al'_0={\rm id}$, $\al'_i=c_i\al_i \in \mathcal{A} \setminus \mathcal{C} = \mathcal{A} \setminus \mathcal{A}_4$ for all $1\le i\le r-1$ and $\al'_r=c_r$),
there exist $\al,\al'\in {\cal A}$ such that $(y)\al\phi\al'\in{\cal P}^*$. As in the proof of Corollary \ref{cor:com}, this implies $\phi\not\in{\cal A}$ and hence $\phi \notin {\mathcal C}$ as required.
\end{proof}


The Main Theorem in the introduction is a direct consequence of these last two corollaries.

\section{Proofs of the five propositions}\label{s:5props}

\hspace*{.3cm} The technical details necessary to prove Theorem \ref{thm:Pstable}, namely Propositions \ref{prop:PtoQ}, \ref{prop:A0-A1}, \ref{prop:A1-A2}, \ref{prop:A2-A3}, and \ref{prop:A3-A4}, are contained in this section.  The basic idea is to understand the actions of various automorphisms on $\cal P^*$ and $\cal Q^*$.  First, we show in section \ref{secQP} that the map $\pi \beta$ behaves very nicely in this respect (in particular, it preserves $\cal P^*$).  The subsequent two sections study how affine automorphisms affect things.  The essential idea is that an affine map can distort $\cal P^*$ some, but this can be rectified by subsequent applications of $\pi$ and/or $\beta$.  For technical reasons, we treat triangular affine maps separately in the final section, and non-triangular affine maps in section \ref{nontriangularAffine}.

\subsection{From $\mathcal{Q}_{m,n}^{*}$ to $\mathcal{P}_{m,n}^*$} \label{secQP}

\begin{definition}
Let $\gamma \in \GA_3(\Ik)$, and set $w_1=(4,1,0)$, $w_2=(4,0,1)$, and $w_3=(8,2,1)$. $\gamma$ is called {\em $\beta$-shaped} if $\deg _{w_i} \gamma = \deg _{w_i} \beta$ and $\ldeg _i \gamma = \ldeg _i \beta$ for each $i=1,2,3$.
\end{definition}

\begin{example}
$\beta$ and $\beta ^{-1}=(x-y^2(y-z^2)^2, y-z^2,z)$ are both $\beta$-shaped.
\end{example}

\begin{lemma}\label{lem:PtoQ}
Let $m\ge 1$ and $n\ge 0$ be integers. If $\gamma \in \mathcal{G}$ is $\beta$-shaped, then $(\mathcal{Q}_{m,n}^{*}) \pi\gamma \subset \mathcal{P}_{m,n}^*$.
\end{lemma}

\begin{proof}
Let $P\in\mathcal{Q}_{m,n}^{*}$, and write $\gamma = (X,Y,Z)$.  Then since $\gamma$ is $\beta$-shaped, the degrees of $X$, $Y$, and $Z$ are given in the following table:
$$
\begin{array}{|c|c|c|c|c|c|c|c|c|}
  \hline
    & \deg_{(4,1,0)} & \deg_{(4,0,1)} & \deg_{(8,2,1)} & {\rm ldeg}_1 & {\rm ldeg}_2 & {\rm ldeg}_3 \\
  \hline
  X & 4       & 4       & 8       & (1,0,0) & (0,4,0) & (0,2,4) \\
  \hline
  Y & 1       & 2       & 2       & (0,1,0) & (0,1,0) & (0,0,2) \\
  \hline
  Z & 0       & 1       & 1       & (0,0,1) & (0,0,1) & (0,0,1) \\
  \hline
\end{array}
$$

Since $P \in \mathcal{Q}_{m,n}^*$, then for all $v=(i,j,k)\in\supp(P)$, we have $i+j\le m$, $3i+3j+k\le 3m+n$
and $(\bx^v)\pi\gamma=Y^{i}X^{j}Z^{k}$. We deduce:
\begin{align*}
\deg_{(4,1,0)}((\bx^v)\pi\gamma)&=i+4j\le 4(i+j)\le 4m,\\
\deg_{(4,0,1)}((\bx^v)\pi\gamma)&=2i+4j+k\le (i+j)+(3i+3j+k)\le m+3m+n=4m+n,\\
\deg_{(8,2,1)}((\bx^v)\pi\gamma)&=2i+8j+k\le 5(i+j)+(3i+3j+k)\le 5m+3m+n=8m+n,\\
{\rm ldeg}_2((\bx^v)\pi\gamma)&=(0,i+4j,k)\le_2 (0,4m,n),\\
{\rm ldeg}_3((\bx^v)\pi\gamma)&=(0,2j,2i+4j+k)\le_3 (0,2m,4m+n).
\end{align*}
We note that each of these last two inequalities is an equality if and only if $(i,j,k)=(0,m,n)$ which belongs to $\supp(P)$.  Thus $(P)\pi \gamma \in \mathcal{P}_{m,n}^*$.
\end{proof}

Applying this to $\beta$ and $\beta ^{-1}$, and recalling that $\mathcal{P}^* \subset \mathcal{Q}^*$, we have
\begin{proposition}\label{prop:PtoQ}
If $\gamma\in\{\beta,\beta^{-1}\}$ then $({\cal Q}^*)\pi\gamma\subset{\cal P}^*$ and $({\cal P}^*)\pi\gamma\subset{\cal P}^*$.
\end{proposition}

\subsection{The non triangular case}\label{nontriangularAffine}

The following technical lemma is necessary to prove Proposition \ref{prop:A0-A1}.

\begin{lemma}
Let $(a,b,c)\in\IN^3\pri\{(0,0,0)\}$. We set $f(v)=ai+bj+ck$ for $v=(i,j,k)\in P_{m,n}$.
Set $m'=\max\{f(v)\ |\ v\in P_{m,n}\}$.
\begin{enumerate}[1)]
\item  If $b>\max\{{a\over 4},2c\}$ and $c\ne 0$ then $f(v)=m'$ if and only if $v=(0,4m,n)$.
\item If $b>{a\over 4}$ and $c=0$ then $f(v)=m'$ if and only if $v=(0,4m,d)$ with $0\le d\le n$.
\item If $c>\max\{{b\over 2},{a-2b\over 4}\}$ and $b\ne 0$ then $f(v)=m'$ if and only if $v=(0,2m,4m+n)$.
\item If $c>{a\over 4}$ and $b=0$ then $f(v)=m'$ if and only if $v=(0,d,4m+n)$ with $0\le d\le 2m$.
\item If $c=\frac{a-2b}{4} > \frac{b}{2}$, then $f(v)=m^\prime$ if and only if $v=(m-d,2d,4d+n)$ with $0 \leq d \leq m$.
\end{enumerate}
\end{lemma}

\begin{proof} Let ${\rm conv}$ denote the convex hull in $\IR^3$.  Define $S_1, S_2 \in P_{m,n}$ by
\begin{align*}
S_1=\{ &(0,0,0), (m,0,0), (0,4m,0), (m,0,n), \\ &(0,4m,n), (0,2m,4m+n), (0,0,4m+n)\} \\
S_2=\{&(m,0,n), (0,4m,n), (0,2m,4m+n)\}
\end{align*}
Note that $S_2 \subset S_1$.  It is easy to check (see Figure \ref{figp}) that ${\rm conv}\ P_{m,n} = {\rm conv}\ S_1$ (and in fact, $P_{m,n} = \left( {\rm conv}\ S_1\right)  \cap \IN^3$).  Then, since $f$ is a linear form and $a,b,c \geq 0$ , we have
$$ m^\prime = \max\{f(v)\ |\ v\in {\rm conv}\ P_{m,n}\} = \max\{f(v)\ |\ v\in S_1\} = \max\{f(v)\ |\ v\in S_2\}$$
We deduce
$$m'=\max\{am+cn,4bm+cn,2bm+4cm+cn\}=m\left(\max\{a,4b,2b+4c\}\right)+cn.$$

Cases 1) and 2): If $b>\max\{{a\over 4},2c\}$ then $m^\prime=4bm+cn$, and $E:=\{v\in \IR^3\ | \ f(v)=m'\}$  is a plane which contains $(0,4m,n)$ and neither $(m,0,n)$ nor $(0,2m,4m+n)$.  Since ${\rm conv}\ P_{m,n} = {\rm conv}\ S_1$ and $E$ is not parallel to any face of ${\rm conv}\ S_1$, then $E \cap {\rm conv}\ P_{m,n}$ must be either a single point or an edge in $S_1$; thus we must either have $E \cap P_{m,n}=\{(0,4m,n)\}$ or $E \cap P_{m,n} = \{ (0,4m, d)\ |\ 0 \leq d \leq n\}$.  It is easy to check that the former case happens precisely when $c \neq 0$, and the latter when $c=0$.

Cases 3) and 4): If $c > \max\{ \frac{b}{2}, \frac{a-2b}{4}\}$, then $m^\prime=2bm+4cm+cn$, and $E:=\{v\in \IR^3\ | \ f(v)=m'\}$  is a plane which contains $(0,2m,4m+n)$ and neither $(m,0,n)$ nor $(0,4m,n)$.  Since ${\rm conv}\ P_{m,n} = {\rm conv}\ S_1$and $E$ is not parallel to any face of ${\rm conv}\ S_1$, then $E \cap {\rm conv}\ P_{m,n}$ must be either a single point in $S_1$, or a line segment connecting two points in $S_1$; thus we must either have $E \cap P_{m,n}=\{(0,2m,4m+n)\}$ or $E \cap P_{m,n} = \{ (0,d, 4m+n)\ |\ 0 \leq d \leq 2m\}$.  It is easy to check that the former case happens precisely when $b \neq 0$, and the latter when $b=0$.

Case 5) : If $c=\frac{a-2b}{4}>\frac{b}{2}$, then $m^\prime=ma+cn$, and $E:=\{v \in \IR^3\ |\ f(v)=m^\prime\}$ is a plane containing $(m,0,n)$ and $(0,2m,4m+n)$ but not $(m,0,n)$.  Then $E \cap \conv P_{m,n}$ is the line segment from $(m,0,n)$ to $(0,2m,4m+n)$, giving the result.
\end{proof}

\begin{definition}
A degree function ${\rm deg}$ is called {\em $\beta$-lexicographic} if, writing $\beta=(X,Y,Z)$, ${\rm deg}(X)>{\rm deg}(Y)>{\rm deg}(Z)$.
\end{definition}
\begin{example}
The three lexicographic degrees, as well as the weighted degrees for the weights $(4,1,0)$, $(4,0,1)$, $(8,2,1)$, $(1,1,0)$, and $(3,3,1)$ are all $\beta$-lexicographic, as is the usual degree (i.e. the $(1,1,1)$-weighted degree).
\end{example}

\begin{proposition}\label{prop:A0-A1}
Let $m\ge 1$ and $n\ge 0$ be integers. Let $\alpha\in{\cal A}_0 \setminus {\cal A}_1={\cal A} \setminus {\cal B}$.
Then there exist $m^\prime \geq m$, $n^\prime \geq 0$ such that
$(\mathcal{P}_{m,n}^{*}) \alpha \beta \subset \mathcal{Q}_{m^\prime,n^\prime} ^*$.
\end{proposition}

\begin{proof}
Let $\deg$ be a $\beta$-lexicographic degree function.  We observe that for any fixed $v=(i,j,k) \in \IN^3$, $\deg\left( (\bx^v)\alpha \beta\right)=\deg(X^{i'}Y^{j'}Z^{k'})$
for some triple $(i^\prime, j^\prime, k^\prime)\in\IN^3$.  Moreover, since $\alpha \in \mathcal{A}$ and  $\alpha\not\in {\cal B}$,
it must be the case that $(i^\prime, j^\prime, k^\prime)$ is one of the following 15 triples: $(i+j+k,0,0)$, $(i+j,k,0)$, $(i+j,0,k)$, $(i+k,j,0)$, $(i+k,0,j)$, $(j+k,i,0)$, $(j+k,0,i)$, $(i,k,j)$, $(j,i,k)$, $(j,k,i)$, $(k,i,j)$, $(k,j,i)$, $(i,j+k,0)$, $(j,i+k,0)$, $(k,i+j,0)$.
We alert the reader that the triple $(i^\prime, j^\prime, k^\prime)$ is determined only by $\alpha$ and is independent of $(i,j,k)$ and the choice of $\beta$-lexicographic degree function $\deg$. We have
\begin{align*}
\deg_{(1,1,0)}(X^{i'}Y^{j'}Z^{k'})&=4i'+j' \\
\deg_{(3,3,1)}(X^{i'}Y^{j'}Z^{k'})&=3(4i'+j')+k' \\
 {\rm ldeg}_2(X^{i'}Y^{j'}Z^{k'}) &=(0,4i'+j',k').
\end{align*}
We can write $4i'+j'=ai+bj+ck=f(i,j,k)$, and $3(4i'+j')+k'=a'i+b'j+c'k=(3a+\epsilon)i+(3b+\mu)j+(3c+\nu)k=g(i,j,k)$
where $a,b,c\in\{0,1,4\}$ and $a'=3a+\epsilon$, $b'=3b+\mu$ and $c'=3c+\nu$ where $\epsilon,\mu,\nu\in\{0,1\}$ (since $k'\in\{0,i,j,k\}$).  We note that $a^\prime$, $b^\prime$, and $c^\prime$ must all be nonzero.
We set
\begin{align*}
m' &=\max\{f(v)\ |\ v\in P_{m,n}\} \\ n'&=\max\{g(v)\ |\ v\in P_{m,n}\}-3m'.
\end{align*}
Now, we consider $P\in \mathcal{P}_{m,n}^{*}$. We recall that this implies $$\{(0,4m,n),(0,2m,4m+2n)\}\subset\supp(P)\subset P_{m,n}.$$
Setting $Q=(P)\alpha\beta$,we will show $Q \in\mathcal{Q}_{m^\prime,n^\prime} ^*$.  Note that $$\deg(Q)= \max _{(i,j,k) \in \supp(P)} \deg(X^{i^\prime}Y^{j^\prime}Z^{k^\prime})$$ where $(i',j',k')$
is one of the 15 triples above where $(i,j,k)\in\supp(P)$.  Since $\supp(P)\subset P_{m,n}$, our definitions of $m^\prime$ and $n^\prime$ immediately imply $\deg_{(1,1,0)}(Q)\le m'$ and $\deg_{(3,3,1)}(Q)\le 3m'+n'$. It remains to check that ${\rm ldeg}_2(Q)=(0,m',n')$; to do so, we show that
there exists a unique $v=(i,j,k)\in\supp(P)$ such that ${\rm ldeg}_2(X^{i'}Y^{j'}Z^{k'})=(0,m',n')$.  Equivalently, we show that there exists a unique $v=(i,j,k) \in \supp(P)$ such that $f(v)=m^\prime$ and $g(v)=3m^\prime+n^\prime$. \\

Case 1) Suppose $(i',j',k')\in\{(i+j,k,0), (i+j,0,k), (j,i,k), (j,k,i), (j,i+k,0)\}$. In this case $b=4$
(since $j$ appears in the first component) and $c\le 1$ (since $k$ does not appear in the first component).
We deduce $b>\max\{{a\over 4},2c\}$ and $b'>\max\{{a'\over 4},2c'\}$.  Note that $c^\prime \neq 0$, so applying the preceding lemma (Case 1) to $g$  yields that $v=(0,4m,n)$ is the unique element in $P_{m,n}$, hence in $\supp(P)$, such
that $g(v)=3m'+n^\prime$.  But the lemma (Case 1 or 2) applied to $f$ implies $f(0,4m,n)=m^\prime$ as required.  \\

Case 2) Suppose $(i',j',k')\in\{(i+j+k,0,0), (i+k,j,0), (j+k,i,0), (j+k,0,i), (k,j,i),(k,i+j,0), (i+k,0,j), (k,i,j) \}$.
In this case $c=4$ (since $k$ appears in the first component) .
We deduce $c>\max\{{b\over 2},{a-2b\over 4}\}$ and $c'>\max\{{b'\over 2},{a'-2b'\over 4}\}$.  Since $b^\prime \neq 0$,
the preceding lemma (Case 3) applied to $g$ gives that $v=(0,2m,4m+n)$ is the unique element in $P_{m,n}$, hence in $\supp(P)$, such that $g(v)=3m'+n^\prime$.  But the lemma (Case 3 or 4) applied to $f$ implies $f(0,2m,4m+n)=m^\prime$ as required.  \\

Case 3) Suppose $(i',j',k') \in \{(i,j+k,0), (i,k,j)\}$. In this case $a^\prime=12$ and $c^\prime=3$ (with $b^\prime \in \{1,3\}$).
We deduce $c'>\max\{{b'\over 2},{a'-2b'\over 4}\}$.
The lemma (Case 3) applied to $g$ yields that $v=(0,2m,4m+n)$ is the unique element in $P_{m,n}$, hence in $\supp(P)$, such that $g(v)=3m'+n^\prime$.  But the lemma applied to $f$ (case 3 or 5) implies $f(0,2m,4m+n)=m^\prime$ as required.
\end{proof}


\subsection{The triangular case}
Once again, we begin with a technical lemma which will aid in the proof of Proposition \ref{prop:A1-A2}
\begin{lemma}\label{lem:TC}
Let $m\ge 1$ and $n\ge 0$ be integers. Let $\gamma=(X,Y,Z)\in {\cal B}$ be such that
${\rm ldeg}_2(X)=(0,b,c)$ where $b\ge 1$ and $c\ge 0$, $\deg_{(3,3,1)}(X)=3b+c$ and $\deg_{(3,3,1)}(Y)=3$.
Then $(\mathcal{P}_{m,n}^{*})\pi\gamma\subset {\mathcal{Q}_{m',n'}^*}$ where $m'=4bm\ge m$ and $n'=4cm+n$.
\end{lemma}
\begin{proof}
Note that the assumptions (particularly $\gamma \in {\cal B}$) immediately imply
\begin{align*}
\deg _{(1,1,0)} (X) &= b & \deg _{(1,1,0)} (Y) &= 1 & \deg _{(1,1,0)} (Z) &= 0 \\
\deg _{(3,3,1)} (X) &= 3b+c & \deg _{(3,3,1)} (Y) &=3 & \deg _{(3,3,1)} (Z) &= 1 \\
\ldeg _2 (X) &= (0,b,c) & \ldeg _2 (Y) &= (0,1,0) & \ldeg _2 (Z) &= (0,0,1)
\end{align*}
Let $v=(i,j,k) \in P_{m,n}$.  Noting that $(\bx^v)\pi\gamma=Y^{i}X^{j}Z^{k}$, we compute
\begin{align*}
\deg _{(1,1,0)} ( (\bx)^v \pi \gamma) &= i+bj \\&\leq b(4i+j) \\&\leq 4bm=m^\prime \\
\deg _{(3,3,1)} ( (\bx)^v \pi \gamma) &= 3i+(3b+c)j+k \\&\leq (3b+c-2)(4i+j)+(8i+2j+k)
\\&\leq (3b+c-2)(4m)+(8m+n) \\
&= 3m^\prime+n^\prime \\
\ldeg _2 ((\bx)^v \pi \gamma) &= (0, i+bj, cj+k) \\
&\leq _2 (0, b(4i+j), n^\prime) \\
&\leq _2 (0,m^\prime,n^\prime)
\end{align*}
Moreover, equality is attained in the last case if and only if $(i,j,k)=(0,4m,n)$, which is in the support of every element of $\mathcal{P}_{m,n}^*$.  Thus we see $(\mathcal{P}_{m,n}^{*})\pi\gamma\subset\mathcal{Q}_{m',n'}^*$ as desired.
\end{proof}

%
%

\begin{proposition}\label{prop:A1-A2}
If $\al\in{\cal A}_1\pri{\cal A}_2$ then $({\cal P}^*)\pi\beta^{-1}\al\beta\subset{\cal Q}^*$.
\end{proposition}

\begin{proof} We write $\al=(a_1x+b_1y+c_1z+d_1,b_2y+c_2z+d_2,c_3z+d_3)$ with $a_1,b_2,c_3\in\Ik^*$ and $b_1,c_1,d_1,c_2,d_2,d_3\in\Ik$.  Set $\gamma = \beta ^{-1} \alpha \beta$, and write $\gamma=(X,Y,Z)$.  We will show that  we can apply Lemma~\ref{lem:TC} to $\gamma$.  A direct computation shows
\begin{align*}
Z&=c_3z+d_3 \\
Y&=b_2y+ez^2+fz+g \\
X&=a_1x+
F_4y^4+F_3y^3+F_2y^2+F_1y+F_0
\end{align*}
where
\begin{align*}
F_4 &= a_1-b_2^4 & F_3 &= 2(a_1z^2-b_2^3(Z_2+Z_3))  \\
F_2 &= a_1z^4-b_2^2(Z_2^2+4Z_2Z_3+Z_3^2) & F_1 &= b_1-2b_2(Z_2+Z_3)Z_2Z_3 \\
 F_0 &=Z_1-Z_2^2Z_3^2,
\end{align*}
\begin{align*}
e&=b_2-c_3^2 & f&=c_2-2c_3d_3 & g&=d_2-d_3^2,
\end{align*}
 and
\begin{align*}
Z_1& =b_1z^2+c_1z+d_1 & Z_2&=b_2z^2+c_2z+d_2 & Z_3&=ez^2+fz+g .
\end{align*}
We easily check that $\deg_{(3,3,1)}(Y)=3$. Set ${\rm ldeg}_2(X)=(0,b,c)$ where $b,c\ge 0$ (and in fact, $b \leq 4$).  It remains to be checked that $b \geq 1$ and $\deg _{(3,3,1)} (X) = 3b+c$.

Since $Z_1,Z_2$ and $Z_3$ are polynomials in $z$ of degree $\le 2$, the support of $X$ contains $(1,0,0)$
and some points $(0,j,k)$ such that $2j+k\le 8$ and $j\le b$.
Hence $$\deg_{(3,3,1)}(X)=\max\{3j+k\,;\,(0,j,k)\in\supp(X)\}\le 8+b.$$

We now examine the possible cases.  First, suppose $a_1\ne b_2^4$.  Then $F_4 \neq 0$, so $(b,c)=(4,0)$ and $\deg_{(3,3,1)}(X)=12=3b+c$, and Lemma \ref{lem:TC} completes the proof.

Next, assume for the remainder that $a_1=b_2^4$, so $F_4=0$.  If $e \neq 0$, then $F_3$ is a degree 2 polynomial in $z$, hence $(b,c)=(3,2)$ and $\deg_{(3,3,1)}(X)=11=3b+c$, and again Lemma \ref{lem:TC} completes the proof..

So we may now assume $e=0$ (so $Z_3=fz+g$ and $F_3=-2b_2^3 \left((c_2+f)z+2(d_2+g)\right)$) as well.  If $c_2+f \neq 0$, then $F_3$ is a degree 1 polynomial in $z$,  hence $(b,c)=(3,1)$ and $\deg_{(3,3,1)}(X)=10=3b+c$, and again Lemma \ref{lem:TC} completes the proof.

We now additionally assume $c_2+f=0$ (so $F_3=2(d_2+g)$) as well.  If $d_2+g \neq 0$, then $(b,c)=(3,0)$.  In this case, one can check that $F_2$ is a polynomial of degree at most 3, so $\deg _{(3,3,1)}(X)=9=3b+c$.

Next, we assume that $d_2+g=0$ as well (so $F_3=0$).  Then $F_2$ is a polynomial in $z$ of degree at most 3.  If $c_2+2f \neq 0$, $F_2$ has degree $3$, so $(b,c)=(2,3)$ and $\deg _{(3,3,1)}(X)=9=3b+c$.

Finally, we must consider the case $c_2+2f=0$.  Since we were already assuming $c_2+f=0$, we have $c_2=0$ and $f=0$, hence $d_3=0$.  But since $0=d_2+g=d_2+(d_2-d_3)^2=2d_2$, we have $d_2=0$.  Since $e=0$, we have $b_2=c_3^2$, and since $a_1=b_2^4$, we thus can simplify $\alpha$ to
$$\alpha = (c_3^8x+b_1y+c_1z+d_2, c_3^2y, c_3z)$$
and note that this precisely means that $\alpha \in \mathcal{A}_2$.  In particular, since $\alpha \notin \mathcal{A}_2$, we must have $b >1$.  Thus the result follows immediately from Lemma \ref{lem:TC}

\end{proof}

\begin{proposition}\label{prop:A2-A3}
If $\al\in{\cal A}_2\pri{\cal A}_3$ then $({\cal P}^*)\pi\beta^{-1}\al\beta\pi\beta\subset{\cal Q}^*$.
\end{proposition}

\begin{proof}
Since $\alpha \in \mathcal{A}_2$, we can write $\al=(u^8x+b_1y+c_1z+d_1,u^2y,uz)\in{\cal A}_2\pri{\cal A}_3$
for some $u,b_1\in\Ik^*$ and $c_1,d_1\in\Ik$.  A direct computation shows that $$\pi\beta^{-1}\al\beta\pi\beta=(u^2X,u^8Y+b_1X+b_1z^2+c_1z+d_1,uz)$$ where $X=(x)\beta = x+y^2(y+z^2)^2$ and
$Y=(y)\beta = y+z^2$. Since $b_1$ is nonzero, the rest of the proof proceeds exactly along the lines of the case $(i',j',k')=(i+j,0,k)$ (Case 1) in Proposition~\ref{prop:A0-A1}.
\end{proof}

\begin{proposition}\label{prop:A3-A4} If $\al\in{\cal A}_3\pri{\cal A}_4$ then
$({\cal P}^*)(\pi\beta^{-1})^2\al\beta\pi\beta\subset{\cal Q}^*$.
\end{proposition}

\begin{proof}
Since $\alpha \in \mathcal{A}_3$, we can write $\al=(u^8x+c_1z+d_1,u^2y,uz)$ for some $u\in\Ik^*$ and $c_1,d_1\in\Ik$.
We compute $$\gamma:=\pi\beta^{-1}\al\beta\pi=(u^2x,u^8y+c_1z+d_1,uz).$$
Clearly $\gamma\in{\cal A}_1$, and one easily checks that $\gamma \notin {\cal A}_2$.  Then Proposition~\ref{prop:A1-A2} implies that
$({\cal P}^*)(\pi\beta^{-1})^2\al\beta\pi\beta=({\cal P}^*)\pi\beta^{-1}\gamma\beta\subset{\cal Q}^*$.
\end{proof}


\begin{thebibliography}{99}



\bibitem{Bodnarchuk} Y. Bodnarchuk, {\it On generators of the tame invertible polynomial maps group},
 Internat. J. Algebra Comput, \textbf{15} (2005), no. 5-6, 851–-867.

\bibitem{E} E. Edo, {\it Coordinates of R[x,y]: Constructions and Classifications},
Comm. Algebra,  \textbf{12} (2013) Vol. 41, 4694--4710

\bibitem{EK} E. Edo, S. Kuroda, {\it Generalisations of the tame automorphisms over a domain of positive characteristic},
Transform. Groups, arXiv:1309.2377.

\bibitem{vdE} A. van den Essen, {\it Polynomial Automorphisms
and the Jacobian Conjecture,} Birkhauser Verlag, Basel-Boston-Berlin (2000).




\bibitem{Jung}   H. Jung,  {\it \"{U}ber ganze birationale {T}ransformationen der {E}bene}, {J. Reine Angew. Math.}, \textbf{184} (1942), 161--174.

\bibitem{Kulk} W. van der Kulk, {\it On polynomial rings in two variables}, Nieuw. Arch. Wisk. (3) {\bf 1} (1953), 33--41.




\bibitem{SU} I. Shestakov and U.Umirbaev, {\it The tame and the wild automorphisms of polynomial rings in three variables}, {J. Amer. Math. Soc.}, \textbf{17}  (2004),  no. 1, 197--227.


\bibitem{Wright} D. Wright, {\it The amalgamated product structure of the tame automorphism group in dimension three}, arXiv:1310.8325 (2013).

\end{thebibliography}
\end{document}